\documentclass[12pt]{amsart}
\usepackage[margin=1in]{geometry}
\usepackage{amssymb,amsfonts,amsmath,mathtools}
\usepackage{color}
\usepackage{soul}

\usepackage{enumerate}
\usepackage{mathrsfs}
\usepackage[unicode]{hyperref}
\usepackage[capitalise]{cleveref}
\usepackage[normalem]{ulem}

\usepackage{constants}
\usepackage{bbm}

\usepackage[unicode]{hyperref}
\hypersetup{colorlinks=true, citecolor=blue, linkcolor=blue, urlcolor=blue, pdfstartview=FitH, pdfauthor=Valentin Blomer and Jesse Thorner, pdftitle=Zeros of automorphic L-functions in families near the critical line}

\newtheorem{theorem}{Theorem}[section]
\newtheorem{corollary}[theorem]{Corollary}

\newtheorem{proposition}[theorem]{Proposition}
\newtheorem{lemma}[theorem]{Lemma}

\newtheorem{question*}{Question}
\newtheorem{problem*}{Problem}

\theoremstyle{definition}

\theoremstyle{remark}
\newtheorem*{remark}{Remark}

\numberwithin{equation}{section}

\crefname{figure}{Figure}{Figures}
\theoremstyle{plain}
\newtheorem*{theorem*}{Theorem}
\crefname{theorems}{Theorem}{Theorems}
\crefname{corollaries}{Corollary}{Corollaries}
\newtheorem*{corollary*}{Corollary}
\crefname{corollaries*}{Corollary}{Corollaries}
\crefname{lemma}{Lemma}{Lemmata}
\crefname{proposition}{Proposition}{Propositions}
\crefname{conjectures}{Conjecture}{Conjectures}
\newtheorem*{conjonjecture*}{Conjecture}
\crefname{conjonjectures*}{Conjecture}{Conjectures}
\crefname{definitions}{Definition}{Definitions}
\crefname{hypotheses}{Hypothesis}{Hypotheses}

\newcommand{\Z}{\mathbb{Z}}
\newcommand{\R}{\mathbb{R}}
\newcommand{\Q}{\mathbb{Q}}

\newcommand{\re}{\mathop{\textup{Re}}}
\newcommand{\im}{\mathop{\textup{Im}}}

\newcommand{\Ad}{\mathrm{Ad}}

\newcommand{\GL}{\mathop{\mathrm{GL}}}

\newcommand{\SL}{\mathop{\mathrm{SL}}}

\newcommand{\N}{\mathrm{N}}

\makeatletter
\DeclareFontFamily{U}  {MnSymbolF}{}
\DeclareSymbolFont{symbolsMN}{U}{MnSymbolF}{m}{n}
\SetSymbolFont{symbolsMN}{bold}{U}{MnSymbolF}{b}{n}
\DeclareFontShape{U}{MnSymbolF}{m}{n}{
    <-6>  MnSymbolF5
   <6-7>  MnSymbolF6
   <7-8>  MnSymbolF7
   <8-9>  MnSymbolF8
   <9-10> MnSymbolF9
  <10-12> MnSymbolF10
  <12->   MnSymbolF12}{}
\DeclareFontShape{U}{MnSymbolF}{b}{n}{
    <-6>  MnSymbolF-Bold5
   <6-7>  MnSymbolF-Bold6
   <7-8>  MnSymbolF-Bold7
   <8-9>  MnSymbolF-Bold8
   <9-10> MnSymbolF-Bold9
  <10-12> MnSymbolF-Bold10
  <12->   MnSymbolF-Bold12}{}
\DeclareMathSymbol{\tbigtimes}{\mathop}{symbolsMN}{2}
\newcommand*{\bigtimes}{%
  \DOTSB
  \tbigtimes
  \slimits@ 
}
\makeatother

\renewcommand{\tilde}{\widetilde}

\renewcommand{\epsilon}{\varepsilon}

\DeclareMathAlphabet{\mathpzc}{OT1}{pzc}{m}{it}

\usepackage[utf8]{inputenc}
\usepackage[T1]{fontenc}

\newconstantfamily{abcon}{symbol=C}

\makeatletter
\let\@wraptoccontribs\wraptoccontribs
\makeatother

\title[Zeros of $L$-functions in families near the critical line]{Zeros of $L$-functions in families near the critical line}
\author{Valentin Blomer}
\address{Mathematisches Institut, Endenicher Allee 60, 53115 Bonn, Germany}
\email{\href{mailto:blomer@math.uni-bonn.de}{blomer@math.uni-bonn.de}}
\author{Jesse Thorner}
\address{Department of Mathematics, University of Illinois, Urbana, IL 61801, USA}
\email{\href{mailto:jesse.thorner@gmail.com}{jesse.thorner@gmail.com}}

\begin{document}

\begin{abstract}
We combine the relative trace formula with analytic methods to obtain zero density estimate for $L$-functions in various families of automorphic representations for $\mathrm{GL}(m)$. Applications include strong bounds for the average analytic rank of these $L$-functions at the central point and average equidistribution results for the imaginary parts of the zeros. 
\end{abstract}


\thanks{The first author was  supported by ERC Advanced Grant  101054336 and Germany's Excellence Strategy grant EXC-2047/1 - 390685813.  The second author is partially supported by the National Science Foundation (DMS-2401311) and the Simons Foundation (MP-TSM-00002484).}

\maketitle

\section{Introduction and statement of results}

\subsection{Zero density estimates} A zero density estimate is a practical substitute for the generalized Riemann hypothesis (GRH). While it does not exclude zeros off the critical line, it shows in a precise quantitative sense that most members of a family of $L$-function do not have zeros with a certain distance to the critical line. A prototype is the Bombieri--Vinograodv theorem that shows that a suitable zero density estimate for Dirichlet $L$-functions governs the average distribution of primes in arithmetic progressions of the same strength as the Riemann hypothesis for Dirichlet $L$-functions.

Around 1990, Sarnak introduced a similar concept for automorphic forms: an automorphic density theorem is a statement that in a  family of automorphic forms only few members can violate the Ramanujan conjecture to a given extent. Such statements are usually an application of a trace or relative trace formula.   

In this paper, we use the machinery of the relative trace formula leading to recent progress on automorphic density theorems and large sieve inequalities to obtain classical zero density estimates for $L$-functions in various families of automorphic representations $\pi$ for $\GL(m)$. We start our discussion with a ``level aspect'' family. Let $m\geq 2$ be an integer, $q\geq 1$ an integer, and $\Gamma_0(q)$ the level $q$ congruence subgroup of $\SL_m(\Z)$, where the last row is congruent to $(0, \ldots, 0, *)$ modulo $q$.  Let $\mathcal{I}\subseteq(0,\infty)$ be a fixed interval of length at least $1$, and let $\mathcal{F}(q) = \mathcal{F}_{\mathcal{I},m}(q)$ be the family of cuspidal automorphic representations generated by Hecke--Maa{\ss} newforms for $\Gamma_0(q)$ whose Laplace eigenvalue lies in $\mathcal{I}$.  Whenever we talk about the family $\mathcal{F}(q)$, all implicit and explicit constants may depend on $\mathcal{I}$ and $m$, which we regard as fixed once and for all.  Let
$$V(q) := [{\rm SL}_m(\Z) : \Gamma_0(q)] = q^{m-1}\prod_{p \mid q} \Big(1 + \frac{1}{p} + \cdots + \frac{1}{p^{m-1}}\Big) \asymp \frac{q^m}{\varphi(q)},$$ 
where $\varphi$ is Euler's totient function.  By Weyl's law, we expect this to be the order of magnitude of $|\mathcal{F}(q)|$.

Let $L(s,\pi)$ be the standard $L$-function associated to $\pi$.  Given $\sigma\geq 0$ and $T\geq 1$, we define
\[
N_{\pi}(\sigma,T)=|\{\rho=\beta+i\gamma\colon \beta\geq \sigma,~|\gamma|\leq T,~L(\rho,\pi)=0\}|,
\]
where the zeros are counted with multiplicity. GRH predicts  $N_{\pi}(\sigma,T)=0$ for $\sigma > 1/2$.  As of now, it is not known whether there exists a constant $\delta>0$ such that each $L(s,\pi)$ for $\pi \in \mathcal{F}(q)$  does not vanish in the region $\re(s)\geq 1-\delta$. Our main result is a spectrally weighted zero density estimate which is strong in the $q$-aspect and at most polynomial in the $T$-aspect.

\begin{theorem}
\label{thm:ZDE}
Let $q\geq 2$ be an integer, $T\geq 2$, and $\sigma\geq 1/2$.  If $0<B<1$, then
\[
\frac{1}{V(q)} \sum_{\pi\in\mathcal{F}(q)}\frac{N_{\pi}(\sigma,T)}{L(1,\pi,\Ad)}\ll_{B} T^{m} q^{-B(\sigma-\frac{1}{2})}F(q)(\log q)^5.
\]
The function $F(q)$ is defined in \eqref{eqn:fq} below and satisfies $F(q) \ll 2^{\omega(q)}$, where $\omega(q)$ is the number of distinct prime factors of $q$. 
\end{theorem}
\begin{remark}
See \eqref{eqn:F(q)_bound} below for the maximal order of $F(q)$.
\end{remark}

The weight $L(1,\pi,\Ad)^{-1}$ is reminiscent of the  Kuznetsov formula. The bound \eqref{eqn:average_weight} below shows that on average $L(1,\pi,\Ad)^{-1}$ is $\ll 1$.  While it is a very hard problem (associated with Landau--Siegel zeros) to obtain individual lower bounds for $L(1,\pi,\Ad)$, relatively strong upper bounds for $L(1,\pi,\Ad)$ are available. In particular, it is much easier to remove the weights than to insert them (should this be required for a certain application). It follows from work of Li \cite{Li} and Bushnell and Henniart \cite{BH} that there exists an effectively computable constant $\Cl[abcon]{unweighted}>0$ such that if
\[
R_{\mathcal{F}}(q)=\max_{\pi\in\mathcal{F}(q)}L(1,\pi,\Ad),
\]
then
\begin{equation}
\label{eqn:Li_Ad}
R_{\mathcal{F}}(q)\ll \exp\Big(\Cr{unweighted}\frac{\log q}{\log\log q}\Big).
\end{equation}
If the generalized Ramanujan conjecture holds for all $\pi\in\mathcal{F}(q)$, we have the stronger bound $R_{\mathcal{F}}(q) \ll (\log q)^{m}$. With this in mind, \cref{thm:ZDE} immediately implies the following unweighted zero density estimate.

\begin{corollary}
\label{cor:ZDE}
 Let $q\geq 2$ be an integer, $T \geq 2$, and $\sigma\geq 1/2$.  If $0<B<1$, then
\[
\frac{1}{V(q)} \sum_{\pi\in\mathcal{F}(q)} N_{\pi}(\sigma,T)\ll_{B} T^{m} q^{-B(\sigma-\frac{1}{2})}R_{\mathcal{F}}(q)F(q)(\log q)^5.
\]
In particular, if $\delta>0$, then for all $\pi\in\mathcal{F}(q)$ with $O(V(q)[T^m R_{\mathcal{F}}(q)F(q)(\log q)^5]^{-\delta})$ exceptions, any nontrivial zero $\beta+i\gamma$ of $L(s,\pi)$ with $|\gamma|\leq T$ satisfies
\[
 \Big|\beta-\frac{1}{2}\Big|\leq \frac{1+\delta}{B}\cdot\frac{\log[R_{\mathcal{F}}(q)F(q)(\log q)^5 T^m]}{\log q}.
\]
\end{corollary}

As mentioned earlier, the proof of Theorem \ref{thm:ZDE} relies among other things on 
 fairly general large sieve inequalities on $\GL(m)$, which we proceed to describe. For $\pi \in \mathcal{F}(q)$, let $B_{\pi}(n_{m-1}, \ldots, n_1)$ denote the arithmetically normalized Fourier coefficient. In particular, for a prime $p \nmid q$, we have the following representation in terms of Schur polynomials in terms of the Satake parameters $\alpha_{j, \pi}(p)$ (see\ Section \ref{sec:L-functions} for details):
\begin{equation}\label{eqn:schur}
B_{\pi}(p^{a_{m-1}}, \ldots, p^{a_{1}}) = \frac{\det[\alpha_{j,\pi}(p)^{a_{m-1} + \cdots + a_{k} + m-k}]_{jk} }{ \det[\alpha_{j,\pi}(p)^{m-k}]_{jk}}.
\end{equation}
Let $\mathbb{N}$ be the set of positive integers.  For $\textbf{n} = (n_{m-1}, \ldots, n_{1}) \in \Bbb{N}^{m-1}$, we define
$$w(\textbf{n}) = n_1 n_2^2 \ldots   n_{m-1}^{m-1},$$
and we write $(\textbf{n}, q) = 1$ to mean that all $n_j$ are coprime to $q$. With this notation, we can state the following result of independent interest which generalizes \cite[Theorem 4]{Blomer_density}. 
\begin{theorem}
\label{thm:MVTgeneral}
Let $q\geq 2$ be an integer and $\beta\colon \mathbb{N}^{m-1}\to\mathbb{C}$ a function.  There exists a constant $\Cl[abcon]{Blomer1}\in(0,1)$ such that if $1\leq x\leq \Cr{Blomer1}q$, then
\[
\frac{1}{V(q)} \sum_{\pi\in\mathcal{F}(q)}\frac{1}{L(1,\pi,\Ad)}\Big|\sum_{\substack{w(\textbf{n})\leq x \\ (\textbf{n},q)=1}}\beta(\textbf{n})B_{\pi}(\textbf{n})\Big|^2\ll \sum_{\substack{w(\textbf{n}) \leq x \\ (\textbf{n},q)=1}}|\beta(\textbf{n})|^2.
\]
\end{theorem}

The proof is an application of the $\GL(m)$ Kuznetsov formula. 
In Section \ref{sec3}, we will apply this in a number of concrete cases related to linear forms in Dirichlet coefficients of $L(s, \pi)^{-1}$ and $L'/L(s, \pi)$ that cannot be expressed in terms of standard Hecke eigenvalues $\lambda_{\pi}(n) = B_{\pi}(1, \ldots, 1, n)$, so that for instance  \cite[Theorem 4]{Blomer_density} would not be applicable.

\medskip

The analytic technique described in Section \ref{sec:ZDE} can also be applied to other families if large sieve inequalities analogous to \cref{thm:MVTgeneral} are available.  For example, following Jana \cite{Jana_newvectors}, we consider the ``spectral'' family $\mathcal{G}(Q)$ of  level $1$  cuspidal automorphic representations $\pi$ of $\mathrm{PGL}_m(\Z)$ with analytic conductor $C(\pi)\leq Q$, with $C(\pi)$ as defined by Iwaniec and Sarnak \cite{IS}, cf.\ \eqref{eqn:analytic_conductor_def} below. Again, we regard $m$ as fixed and all explicit and implicit constants may depend on $m$.  The same method of proof shows the following. 
\begin{theorem}
\label{thm:ZDE2}
Let $Q, T\geq 2$ and $\sigma\geq 1/2$.  There exists an effectively computable constant $\Cl[abcon]{TRange_Jana2}>0$ such that if $0<B<1$, then
\[
\frac{1}{ |\mathcal{G}(Q)|} \sum_{\pi\in\mathcal{G}(Q)}\frac{N_{\pi}(\sigma,T)}{L(1,\pi,\Ad)}\ll_{B} T^{m}Q^{-B(\sigma-\frac{1}{2})}(\log Q)^5
\]
and
\[
\frac{1}{ |\mathcal{G}(Q)|} \sum_{\pi\in\mathcal{G}(Q)} N_{\pi}(\sigma,T) \ll_{B}  T^{m}Q^{-B(\sigma-\frac{1}{2})}R_{\mathcal{G}}(Q)(\log Q)^5,
\]
where
\begin{equation}
\label{eqn:Li_Ad_spectral}
R_{\mathcal{G}}(Q) = \max_{\pi\in\mathcal{G}(Q)}L(1,\pi,\Ad) \ll \exp\Big(\Cr{TRange_Jana2}\frac{\log Q}{\log\log Q}\Big).
\end{equation}
\end{theorem}

\cref{thm:ZDE}, \cref{cor:ZDE}, and \cref{thm:ZDE2} show their strength for $\sigma$ close to $1/2$.  They complement work of Humphries and Thorner \cite[Theorem 1.1]{HT_density} which is tailored for  $\sigma$ close to 1: If $T\geq 1$, $\sigma\geq 1/2$, and $\epsilon>0$, then
\begin{equation}
\label{eqn:HT_GLn}
\sum_{\pi\in\mathcal{F}(q)}N_{\pi}(\sigma,T)\ll_{ \epsilon} (qT)^{7.15m^2(1-\sigma)+\epsilon},\qquad \sum_{\pi\in\mathcal{G}(Q)}N_{\pi}(\sigma,T)\ll_{\epsilon} (QT)^{7.15m^2(1-\sigma)+\epsilon}.
\end{equation}
It follows from \eqref{eqn:HT_GLn} (and a rescaling of $\epsilon$) that for all $\pi\in\mathcal{F}(q)$ with $O_{\epsilon}(q^{\epsilon})$ exceptions, $L(s,\pi)$ does not vanish when $\re(s)\geq 1-\epsilon/(9 m^2)$ and $|\im(s)|\leq q$.  A similar statement holds for $\mathcal{G}(Q)$.

\subsection{Applications}

 Our new zero density estimates are motivated by their arithmetic consequences, a few of which are described below.  As we will soon see, the quality of these arithmetic consequences is highly sensitive to the specific shape of the zero density estimate; in particular, we cannot afford powers of $q^{\varepsilon}$ in our density results.

\subsubsection{Order of vanishing at $s=1/2$}

It is generally believed that $L$-functions should not vanish at $s=\frac{1}{2}$ unless $\pi$ is self-dual and the sign of the functional equation is $-1$, or there is a special arithmetic reason, like the relationship with Hasse--Weil $L$-functions and ranks of elliptic curves predicted by the Birch and Swinnerton-Dyer conjecture. See \cite{RadziwillYang} for higher rank examples. 

  If $\pi\in\mathcal{F}(q)$, then GRH implies that
\begin{equation}
\label{eqn:GRH}
\mathop{\mathrm{ord}}_{s=1/2}L(s,\pi)\ll\frac{\log q}{\log\log q}.	
\end{equation}
Using   \cref{thm:ZDE} and \cref{thm:MVTgeneral}, we prove the following result.
\begin{theorem}
\label{thm:order_of_vanishing}
If $q\geq 3$ is an integer, then
\[
\frac{1}{V(q)}\sum_{\pi\in\mathcal{F}(q)}\displaystyle\frac{\mathop{\mathrm{ord}}_{s=1/2}L(s,\pi)}{L(1,\pi,\Ad)}\ll  \log F(q)+\log\log q,\quad \frac{1}{V(q)}\sum_{\pi\in\mathcal{F}(q)}\mathop{\mathrm{ord}}_{s=1/2}L(s,\pi)\ll  \frac{\log q}{\log\log q}.
\]
\end{theorem}
\cref{thm:order_of_vanishing} demonstrates that on average over $\pi\in\mathcal{F}(q)$, the GRH-conditional bound \eqref{eqn:GRH} is unconditionally true.  When the weights $L(1,\pi,\Ad)^{-1}$ are inserted, we obtain on average a bound much stronger than what GRH implies for an individual $L$-function. If the generalized Ramanujan conjecture were true, then our proof would produce matching weighted and unweighted bounds for the average order of vanishing.

It is worth noting that while we could state our zero density estimates in \cref{thm:ZDE,cor:ZDE} coarsely as $\ll_{B,\epsilon} q^{-B(\sigma-1/2)+\varepsilon}$, the corresponding upper bounds in \cref{thm:order_of_vanishing} would only assume the form $o(\log q)$.  It is a nice coincidence that the size of Li's bound \eqref{eqn:Li_Ad} is no more than the exponential of the right hand side of \eqref{eqn:GRH}, which we crucially exploit in the unweighted result to obtain a GRH-quality bound on average. 

\cref{thm:order_of_vanishing} is new for all $m\geq 2$.  Kowalski and Michel \cite{KM_J0} (see also \cite{KMV}) proved a stronger result as $\pi$ varies in the family of cuspidal automorphic representations generated by the holomorphic elliptic curve newforms of weight $2$ and prime level $q$.  In this family, the average order of vanishing at $s=\frac{1}{2}$ is $O(1)$.   Their proof relies on a version of \cref{thm:ZDE} for prime $q$ where the power of $\log q$ equals $1$.  In light of the Riemann--von Mangoldt asymptotic for the count of zeros up to height $T$, this power of $\log q$, if attainable, is optimal.

A similar result with essentially the same proof holds for the spectral family $\mathcal{G}(Q)$ featured in Theorem \ref{thm:ZDE2}. 
\begin{theorem}
\label{thm:ZDE2appl}
If $Q\geq 3$, then
\[
\frac{1}{|\mathcal{G}(Q)|}\sum_{\pi\in\mathcal{G}(Q)}\frac{\mathop{\mathrm{ord}}_{s=1/2}L(s,\pi)}{L(1,\pi,\Ad)}\ll \log\log Q,\qquad \frac{1}{|\mathcal{G}(Q)|}\sum_{\pi\in\mathcal{G}(Q)}\mathop{\mathrm{ord}}_{s=1/2}L(s,\pi)\ll \frac{\log Q}{\log\log Q}.
\]
\end{theorem}

\subsubsection{Equidistribution of fractional imaginary parts of nontrivial zeros}

 Fix $\alpha\in\R^{\times}$, and let $I\subseteq[0,1]$ be a subinterval.  Let $\{x\}$ denote the fractional part of $x$.
 
 For the Riemann zeta function, let $N(T)$ be the count (including multiplicity) of nontrivial zeros $\beta+i\gamma$ with $|\gamma|\leq T$.  Refining preceding work of Rademacher \cite{Rademacher} and Hlawka \cite{Hlawka}, Fujii \cite{Fujii} proved that
\begin{equation}
\label{eqn:Fujii}
\sup_{I\subseteq[0,1]}\Big|\frac{1}{N(T)}\sum_{\substack{|\gamma|\leq T \\ \{\alpha\gamma\}\in I}}1-|I|\Big|\ll_{\alpha} \frac{\log\log T}{\log T}.
\end{equation}
Thus, the fractional parts $\{\alpha\gamma\}$ with $|\gamma|\leq T$ are equidistributed in intervals larger than $(\log\log T)/\log T$.  Ford and Zaharescu \cite{FordZaharescu} proved that there exists an explicit function $g_{\alpha}(t)$ such that if the indicator function of $I$ is replaced with a twice continuously differentiable weight function $h$, then
\begin{equation}
\label{eqn:FordZaharescu}
\sum_{|\gamma|\leq T}h(\{\alpha \gamma\})=N(T)\int_{0}^{1}h(t)dt + T\int_{0}^{1}h(t)g_{\alpha}(t)dt+o_{\alpha,h}(T).
\end{equation} 
Ford, Soundararajan, and Zaharescu \cite{FSZ} related this phenomenon to Montgomery's pair correlation function and the distribution of primes in short intervals.

This phenomenon conditionally extends to other $L$-functions.  Let $\mathfrak{F}_m$ be the set of cuspidal automorphic representations of $\GL_m(\mathbb{A}_{\Q})$ with unitary central character.  It follows from the work of Ford, Soundararajan, and Zaharescu \cite{FSZ} that if there exist constants $A_{\pi}>0$ and $B_{\pi}>0$ such that the zero density estimate
\begin{equation}
\label{eqn:ZDE_conj}
N_{\pi}(\sigma,T)\ll_{\pi} T^{1-A_{\pi}(\sigma-\frac{1}{2})}(\log T)^{B_{\pi}},\qquad \textup{$T\geq 1$ and $\sigma\geq0$}
\end{equation}
holds, then analogues of \eqref{eqn:Fujii} and  \eqref{eqn:FordZaharescu} hold for the zeros of $L(s,\pi)$.  So far, \eqref{eqn:ZDE_conj} is known only when $m=1$ \cite{Fujii_density,Selberg_density} or $m=2$ \cite{BLTZ,Luo_density}.

For $m\geq 3$, we use \cref{thm:ZDE} to prove that on average over $\pi\in\mathcal{F}(q)$ (weighted by $L(1,\pi,\Ad)^{-1}$), the fractional imaginary parts are equidistributed up to a height depending on $q$.  To state our result, we write $N_{\pi}(T)=N_{\pi}(0,T)$.  Recall \cite[Theorem 5.8]{IK} that
\begin{equation}
\label{eqn:RvonM}
N_{\pi}(T)=\frac{T}{\pi}\log\Big(q_{\pi} \Big(\frac{T}{2\pi e}\Big)^m\Big)+O(\log(C(\pi)T)),
\end{equation}
where $q_{\pi}$ (resp.\ $C(\pi)$) is the arithmetic conductor (resp.\ analytic conductor) of $\pi$.

\begin{theorem}
\label{thm:zeros_equidistr}
Let $q$ be a large prime and $\alpha\in\R^{\times}$.  If $r>2$, $T=(\log q)^r$, and $I$ denotes a subinterval of $[0,1]$, then
\[
\sup_{I\subseteq[0,1]}\Big|\frac{1}{V(q)}\sum_{\pi\in\mathcal{F}(q)}\frac{1}{L(1,\pi,\Ad)}\Big(\frac{1}{N_{\pi}(T)}\sum_{\substack{|\gamma|\leq T \\ \{\alpha\gamma\}\in I}}1-|I|\Big)\Big|\ll_{\alpha,r}\frac{\log\log\log q}{\log\log q}.
\]	
\end{theorem}
  
Because $T$ must be so small compared to $q$, our proof cannot recover any secondary main terms in \eqref{eqn:FordZaharescu}.  Unlike \cref{thm:order_of_vanishing}, our proof of \cref{thm:zeros_equidistr} is so sensitive to the factor of $q^{-B(\sigma-1/2)}$ in our zero density estimates that we cannot afford to have any factor larger than a power of $\log q$ (hence our constraint on $T$ in \cref{thm:zeros_equidistr}).  In particular, we cannot afford a factor as large as \eqref{eqn:Li_Ad}, and we must choose $q$ so that $F(q)$ is small (e.g., $q$ a prime).

A similar result with essentially the same proof holds for the spectral family $\mathcal{G}(Q)$ featured in Theorem \ref{thm:ZDE2}.

\begin{theorem}
\label{thm:zeros_equidistr2}
 Let $Q$ be large and $\alpha\in\R^{\times}$.  If $r>2$, $T=(\log Q)^r$, and $I$ denotes a subinterval of $[0,1]$, then
\[
\sup_{I\subseteq[0,1]}\Big|\frac{1}{|\mathcal{G}(Q)|}\sum_{\pi\in\mathcal{G}(Q)}\frac{1}{L(1,\pi,\Ad)}\Big(\frac{1}{N_{\pi}(T)}\sum_{\substack{|\gamma|\leq T \\ \{\alpha\gamma\}\in I}}1-|I|\Big)\Big|\ll_{\alpha,r}\frac{\log\log\log Q}{\log\log Q}.
\]
\end{theorem}

\section{Properties of \texorpdfstring{$L$}{L}-functions}
\label{sec:L-functions}


Let $\pi$ be a cuspidal automorphic representation for ${\rm PGL}_m(\Bbb{A}_{\Bbb{Q}})$ generated by a cuspidal Hecke-Maa{\ss} newform for the subgroup $\Gamma_0(q) \subseteq {\rm SL}_m(\Bbb{Z})$.  By \cite{JPSS-MathAnn},  the arithmetic conductor of $\pi$ is $q$.  Its local $L$-function $L(s,\pi_{p})$ is defined in terms of the Satake parameters $\alpha_{1,\pi}(p),\ldots,\alpha_{m,\pi}(p)\in\mathbb{C}$ by
\begin{equation*}
	L(s,\pi_{p})=\prod_{j=1}^{m}(1-\alpha_{j,\pi}(p)p^{-s})^{-1}=\sum_{k=0}^{\infty}\frac{\lambda_{\pi}(p^k)}{p^{ks}}.
\end{equation*}
Since our representations have trivial central character, we have
\begin{equation}
\label{product}
\prod_{j=1}^m \alpha_{j, \pi}(p) = 1
\end{equation}
for $p \nmid q$.  The standard $L$-function $L(s,\pi)$ associated to $\pi$ is then
\[
L(s,\pi)=\prod_{p} L(s,\pi_{p})=\sum_{n=1}^{\infty}\frac{\lambda_{\pi}(n)}{n^s}.
\]
The Euler product and Dirichlet series converge absolutely when $\re(s)>1$.  For $(n, q) = 1$, the number $\lambda_{\pi}(n)\in\mathbb{C}$ is the $n$-th Hecke eigenvalue of $\pi$ and equals the normalized Fourier coefficient $B_{\pi}(1, \ldots, 1, n)$.  We recall the Schur polynomials given in \eqref{eqn:schur} for general Fourier coefficients.
 
 At the archimedean place, there are $m$ Langlands parameters $\kappa_{\pi}(1),\ldots,\kappa_{\pi}(m)\in\mathbb{C}$ such that
\[
L(s,\pi_{\infty}) = \prod_{j=1}^{n}\Gamma_{\R}(s+\kappa_{\pi}(j)),\qquad \Gamma_{\R}(s) = \pi^{-\frac{s}{2}}\Gamma\Big(\frac{s}{2}\Big).
\]
By combining the work in \cite{LRS} for unramified primes and \cite{MS} for ramified primes (cf.\ also \cite{BB-Bulletin}), we know that there exists
\begin{equation}
\label{eqn:ramanujan_progress}
0\leq\theta_m\leq \frac{1}{2}-\frac{1}{m^2+1}
\end{equation}
such that
\begin{equation}
\label{eqn:LRS_finite}
	|\alpha_{j,\pi}(p)|\leq  p^{\theta_m}\qquad\textup{ and }\qquad\re(\kappa_{\pi}(j))\geq -\theta_m.
\end{equation}
It follows that if $\tau_m(n)$ is the $n$-th Dirichlet coefficient of $\zeta(s)^m$, then
\begin{equation}
\label{eqn:lambda_bound}
|\lambda_{\pi}(n)|\leq \tau_m(n)n^{\theta_m}.
\end{equation}
The generalized Ramanujan conjecture asserts that in \eqref{eqn:ramanujan_progress}, one may take $\theta_m=0$.

Let $\tilde{\pi}$ be contragredient to $\pi$. The completed $L$-function
\begin{equation}
\label{eqn:completed_L-function}
\Lambda(s,\pi) = q^{s/2}L(s,\pi)L(s,\pi_{\infty})
\end{equation}
is entire of order 1, and there exists a complex number $W(\pi)$ of modulus 1 such that for all $s\in\mathbb{C}$, we have the functional equation
\begin{equation}
\label{eqn:FE}
\Lambda(s,\pi)=W(\pi)\Lambda(1-s,\widetilde{\pi}).
\end{equation}
The analytic conductor of $\pi$ \cite{IS} is given by
\begin{equation}
\label{eqn:analytic_conductor_def}
C(t,\pi)= q\prod_{j=1}^n(3+|it+\kappa_{\pi}(j)|),\qquad C(\pi)= C(0, \pi).
\end{equation}
We recall the convexity bound
\begin{equation}
\label{eqn:convexity}
L(\tfrac{1}{2}+it,\pi)\ll q^{1/4}(|t|+1)^{m/4}
\end{equation}
which follows for instance from   \cite[Corollary 2.7]{ST}. 

We define the numbers $\mu_{\pi}(n)$ by the (absolutely convergent) Dirichlet series identity
\[
\frac{1}{L(s,\pi)} = \prod_p \prod_{j=1}^m (1-\alpha_{j,\pi}(p)p^{-s}) = \sum_{n=1}^{\infty}\frac{\mu_{\pi}(n)}{n^s},\qquad \re(s)>1.
\]
Using \eqref{eqn:LRS_finite}, we find that
\begin{equation}
\label{eqn:mu_bound}
|\mu_{\pi}(n)|\leq \tau_m(n)n^{\theta_m}.
\end{equation}
At a prime $p \nmid q$, we can use \eqref{eqn:schur} to express the local Euler factor   in terms of Fourier coefficients as 
\begin{equation*}
\frac{1}{L(s, \pi_p)} = 1 - \frac{B_{\pi}(1, \ldots, 1, p)}{p^s} + \frac{B_{\pi}(1, \ldots, 1, p, 1)}{p^{2s}}- \cdots + (-1)^{m-1} \frac{B_{\pi}(p, 1, \ldots, 1)}{p^{(m-1)s}} +\frac{ (-1)^m}{p^{ms}},
\end{equation*}
so that 
\begin{equation}\label{eqn:sothat}
\sum_{(n, q) = 1} \frac{\mu_{\pi}(n)}{n^s} = \sum_{(n_1 \ldots n_m, q) = 1} \frac{\mu^2(n_1\cdots n_m) \mu(n_1n_3\cdots)B_{\pi}(n_{m-1}, \ldots, n_1) }{(n_1 n_2^2 \dots n_m^m)^s}.
\end{equation}

For an integer $n\geq 1$, we define
\begin{equation}
\label{eqn:defa}
a_{\pi}(n) = \begin{cases}
\sum_{j=1}^m \alpha_{j,\pi}(p)^k &\mbox{if $n=p^k$ for some prime $p$ and some integer $k\geq 1$,}\\
0&\mbox{otherwise}
\end{cases}
\end{equation}
and the von Mangoldt function
\[
\Lambda(n) = \begin{cases}
\log p&\mbox{if $n=p^k$ for some prime $p$ and some integer $k\geq 1$,}\\
0&\mbox{otherwise. }
\end{cases}
\]
In the region $\re(s) > 1$, we have the  (absolutely convergent) Dirichlet series identity
\[
-\frac{L'}{L}(s,\pi)=\sum_{n=1}^{\infty}\frac{a_{\pi}(n)\Lambda(n)}{n^{s}}.
\]
Using \eqref{eqn:LRS_finite}, we find that
\begin{equation}
\label{eqn:vonMangoldtBound}
|a_{\pi}(p^k)|\leq m p^{k\theta_m}.
\end{equation}
As an application of the Murnaghan--Nakayama rule,  we can express $a_{\pi}(p^k)$ in terms of Schur polynomials as follows:  for $k \geq 1$ and a prime $p \nmid q$ we have
\begin{equation}\label{eqn:schura}
a_{\pi}(p^k) = B_{\pi}(1, \ldots, 1, p^k) - \sum_{j=2}^{\min(m, k)} (-1)^jB_{\pi}(1, \ldots, 1,\underbrace{ p}_{\text{{\rm at entry }}m-j}, 1, \ldots, 1, p^{k-j})
\end{equation}
with the understanding that for $j=m$ the entry $p$ does not occur.  A proof is given in  \cite[Theorem 7.17.1 with $\mu = \emptyset$]{Stanley}. To translate between our notation and the notation used in \cite{Stanley} note   that our Schur polynomial $B(p^{a_{m-1}}, \ldots, p^{a_1})$ corresponds to $s_{\lambda}(\alpha_{1, \pi}(p), \ldots, \alpha_{m, \pi}(p))$    for the partition $$\lambda = (a_{m-1} + \ldots + a_1,  a_{m-1} + \ldots + a_2, \ldots, a_{m-1})$$ and by \eqref{product} we have $s_{\lambda} = s_{\lambda + (j, \ldots, j)}$ for any partition $\lambda$ and any $j \in \Bbb{N}\cup\{0\}$. The partitions occurring on the right hand side of \eqref{eqn:schura} are precisely the ``hook-shaped'' partitions 
$(k), (k-1, 1), (k-2, 1, 1), \ldots, (\max(k-m+1, 1), 1, \ldots, 1)$. 

In particular, for $p \nmid q$ prime we conclude from \eqref{eqn:schura} that $a_{\pi}(p)=\lambda_{\pi}(p).$

\section{Large sieve inequalities}\label{sec3}

In this section we will specialize \cref{thm:MVTgeneral} in a number of concrete cases and obtain various inequalities of large sieve type. The proof of \cref{thm:MVTgeneral} is given at the end of the section.  We start with a small variation of  \cite[Theorem 4]{Blomer_density}. 

\begin{proposition}
\label{prop:MVT1}
Let $q\geq 2$ be an integer and $\alpha\colon \mathbb{N}\to\mathbb{C}$ a function.  If $1\leq N\leq \Cr{Blomer1}q$, then
\[
\frac{1}{V(q)} \sum_{\pi\in\mathcal{F}(q)}\frac{1}{L(1,\pi,\Ad)}\Big|\sum_{\substack{n\leq N \\ \gcd(n,q)=1}}\alpha(n)\lambda_{\pi}(n)\Big|^2\ll \sum_{\substack{n\leq N \\ \gcd(n,q)=1}}|\alpha(n)|^2.
\]
\end{proposition}
\begin{proof}
This follows from \cref{thm:MVTgeneral} by observing that $\lambda_{\pi}(n) = B_{\pi}(1, \ldots, 1, n)$.
\end{proof}

When $\alpha(1)=1$ and $\alpha(n)=0$ for $n\geq 2$, \cref{prop:MVT1} implies that
\begin{equation}
\label{eqn:average_weight}
\sum_{\pi\in\mathcal{F}(q)}\frac{1}{L(1,\pi,\Ad)}\ll V(q).
\end{equation}
When $\alpha$ is multiplicative and $\alpha(n)\ll 1/\sqrt{n}$, we can remove the coprimality conditions at the cost of introducing a factor of 
\begin{equation}\label{eqn:fq}
F(q) := \prod_{p\mid q}\Big(1+\frac{1}{p^{\frac{1}{2}-\theta_m}}\Big)^m.
\end{equation}
Chebyshev's bound implies that there exists an absolute constant $\Cl[abcon]{Chebyshev}>0$ such that
\begin{equation}
\label{eqn:F(q)_bound}
F(q)\ll \min\Big\{2^{\omega(q)},\exp\Big(\Cr{Chebyshev}m\frac{(\log q)^{1/2+\theta_m}}{(\log\log q)^{1/2+\theta_m}}\Big)\Big\}.
\end{equation}

\begin{corollary}
\label{cor:MVT1}
Let $ q\geq 2$ be an  integer and $\alpha\colon\mathbb{N}\to\mathbb{C}$  a multiplicative function such that $\alpha(n)\ll 1/\sqrt{n}$.  If $1\leq N\leq \Cr{Blomer1}q$, then
\begin{equation}
\label{eqn:main_large_sieve_lambda}
\frac{1}{V(q)}\sum_{\pi\in\mathcal{F} (q)}\frac{1}{L(1,\pi,\Ad)}\Big|\sum_{n\leq N}\alpha(n)\lambda_{\pi}(n)\Big|^2\ll F(q)\log q.
\end{equation}
\end{corollary}
\begin{proof}
For integers $d\geq 1$ and $q\geq 2$, we write $d\mid q^{\infty}$ to denote that there exists $k\in\N$ such that $d\mid q^k$.  With this notation and our assumption that $\alpha$ is multiplicative, the left hand side equals of \eqref{eqn:main_large_sieve_lambda} equals
\begin{equation*}
\begin{aligned}
& \frac{1}{V(q)}\sum_{\pi\in\mathcal{F}(q)}\frac{1}{L(1,\pi,\Ad)}\Big|\sum_{d\mid q^{\infty}}\alpha(d)\lambda_{\pi}(d)\sum_{\substack{n\leq N/d\\ \gcd(n,q)=1}}\alpha(n)\lambda_{\pi}(n)\Big|^2.
\end{aligned}
\end{equation*}
Since by \eqref{eqn:lambda_bound}  we have
  $$\sum_{d\mid q^{\infty}}|\alpha(d)\lambda_{\pi}(d)| \ll \sum_{d\mid q^{\infty}} \frac{\tau_m(d)}{d^{1/2 -\theta_m}} = F(q), $$
an application of \cref{prop:MVT1} yields the bound desired $F(q) \sum_{n \leq N} n^{-1}$.
\end{proof}

This allows a small generalization of \cite[Theorem 5]{Blomer_density} to arbitrary moduli $q$.

\begin{corollary}
\label{cor:2ndmoment}
Fix $0<\epsilon<1$. Uniformly for $t\ll q^{(1-\epsilon)/m}$, we have 
\[
\frac{1}{V(q)} \sum_{\pi\in\mathcal{F}(q)}\frac{|L(\frac{1}{2}+it,\pi)|^2}{L(1,\pi,\Ad)}\ll F(q)\log q. 
\]
\end{corollary}

\begin{proof} This is a standard application of the approximate functional equation exactly as in \cite[Theorem 5]{Blomer_density}. In order to apply \cref{cor:MVT1}, we note that the effective length of the approximate functional equation of $L(1/2+it,\pi)$ is $(q (1 + |t|)^m )^{1/2+\varepsilon/4} \ll q^{1-\epsilon^2/4}$, say, for sufficiently small $\varepsilon > 0$. 
\end{proof}

We also need a large sieve inequality for the numbers $\mu_{\pi}(n)$.

\begin{proposition}
\label{prop:MVT2}
Let $q\geq 2$ be an integer and $\alpha\colon \mathbb{N}\to\mathbb{C}$ a function.  If $1\leq x\leq \Cr{Blomer1}q$, then
\begin{align*}
&\frac{1}{V(q)}\sum_{\pi\in\mathcal{F}(q)}\frac{1}{L(1,\pi,\Ad)}\Big|\sum_{\substack{n\leq x \\ \gcd(n,q)=1}}\alpha(n)\mu_{\pi}(n)\Big|^2\\
&\ll\sum_{\substack{\nu,\tilde{\nu}\leq x^{1/m}\\ (\nu\tilde{\nu}, q) = 1}}\mu(\nu)^2\mu(\tilde{\nu})^2\sum_{\substack{n\leq x/\max(\nu^m, \tilde{\nu}^m) \\ \textup{$n$ is $m$-power free} \\ \gcd(n,\nu \tilde{\nu}q)=1}}\alpha(\nu^m n)\overline{\alpha(\tilde{\nu}^m n)}.
\end{align*}
\end{proposition}

\begin{remark} The right hand side looks more complicated than one would usually expect from a large sieve inequality. The reason is simple. An inspection of \eqref{eqn:sothat} shows that the numbers $\mu_{\pi}(n)$ fail to enjoy any sort of orthogonality on perfect $m$-th powers. Consequently, we cannot expect any better bound than the one given in the corollary.\end{remark}

\begin{proof} In view of \eqref{eqn:sothat},  we have 
$$  \sum_{\substack{n\leq x \\ \gcd(n,q)=1}}\alpha(n)\mu_{\pi}(n) = \sum_{(\nu, q) = 1} \sum_{\substack{w(\textbf{n}) \leq x/\nu^m\\ (\textbf{n}, q) = 1}} \beta(n_{m-1}, \ldots, n_1; \nu)B_{\pi}(n_{m-1}, \ldots, n_1)$$
with 
$$\beta(n_{m-1}, \ldots, n_1; n_m) = \alpha(n_1\cdots n_m^m)\mu(n_1n_3\cdots) \mu^2(\tilde{n}_1\cdots \tilde{n}_m). $$
Therefore, an application of \cref{thm:MVTgeneral} shows 
\begin{displaymath}
\begin{split}
& \frac{1}{V(q)} \sum_{\pi \in \mathcal{F}(q)} \frac{1}{L(1, \pi, \text{Ad})} \Big|\sum_{\substack{n \leq x\\ (n, q) = 1}} \alpha(n) \mu_{\pi}(n)\Big|^2 \\
&\ll    \sum_{\substack{\nu, \tilde{\nu} \leq x^{1/m}\\ (\nu\tilde{\nu}, q) = 1}}  \mu^2(\nu) \mu^2(\tilde{\nu})\sum_{\substack{n_1\cdots n_{m-1}^{m-1} \leq x/\max(\nu^m, \tilde{\nu}^m) \\  \mu^2(n_1\cdots n_m ) =  \mu^2(n_1\cdots n_{m-1}\tilde{n}_m ) = 1\\ (n_1 \cdots n_m \tilde{n}_1 \cdots \tilde{n}_m, q) = 1}}  \alpha(n_1\cdots n_{m-1}^{m-1} \nu^m) \overline{ \alpha(n_1\cdots n_{m-1}^{m-1} \tilde{\nu}^m) }.
\end{split}
\end{displaymath}
Here, $n_1\cdots n_{m-1}^{m-1}$ describes a general $m$-power free number coprime to $\nu\tilde{\nu}q$, and the result follows.
\end{proof}

As with \cref{cor:MVT1}, when $\alpha$ is multiplicative and satisfies $|\alpha(n)|\ll 1/\sqrt{n}$, we can remove the coprimality conditions at the cost of introducing a factor of $F(q)$ as in \eqref{eqn:fq}.

\begin{corollary}
\label{cor:MVT2}
Let $ q\geq 2$ be an integer and $\alpha\colon\mathbb{N}\to\mathbb{C}$ a multiplicative function such that $\alpha(n)\ll 1/\sqrt{n}$.  If $1\leq N\leq \Cr{Blomer1}q$, then
\begin{equation*}
\frac{1}{V(q)} \sum_{\pi\in\mathcal{F} (q)}\frac{1}{L(1,\pi,\Ad)}\Big|\sum_{n\leq N}\alpha(n)\mu_{\pi}(n)\Big|^2\ll  \begin{cases}
  F(q)\log q&\mbox{if $m\geq 3$,}\\
  F(q)(\log q)^3&\mbox{if $m=2$.}
  \end{cases}
\end{equation*}
\end{corollary}
\begin{proof}
We mimic the proof of \cref{cor:MVT1} using the multiplicativity of $\mu_{\pi}(n)$ and $\alpha(n)$ along with \eqref{eqn:mu_bound}. As before we deduce that the left hand side is bounded by
\begin{align*}
&\ll F(q) \sum_{\nu,\tilde{\nu}\leq N^{1/m}} \sum_{ n\leq N/\max(\nu^m, \tilde{\nu}^m)}|\alpha(\nu^m n)\overline{\alpha(\tilde{\nu}^m n)}|\\
& \ll \sum_{\nu,\tilde{\nu}\leq N^{1/m}}\frac{1}{\nu^{m/2}\tilde{\nu}^{m/2}}\log\frac{N}{\max(\nu^{m}, \tilde{\nu}^{m})}\ll\begin{cases}
	 \log N&\mbox{if $m\geq 3$.}\\
	 (\log N)^3&\mbox{if $m=2$,}
	 \end{cases}
\end{align*}
\end{proof}

Finally we record a large sieve inequality for $a_{\pi}(p^k)$ as defined in  \eqref{eqn:defa}. 

\begin{proposition}\label{newcor} Let $q\geq 2$ be an integer and $\alpha\colon \mathbb{N}\to\mathbb{C}$  satisfy $\alpha(n)\ll 1/\sqrt{n}$. If $1\leq x\leq \Cr{Blomer1}q$, then
\begin{align*}
&\frac{1}{V(q)}\sum_{\pi\in\mathcal{F}(q)}\frac{1}{L(1,\pi,\Ad)}\Big|\sum_{ n\leq x }\alpha(n)a_{\pi}(n) \Lambda(n)\Big|^2\ll (\log q)^2. \end{align*}
\end{proposition}

\begin{proof} We split the sum over $n$  into three pieces: 
\begin{itemize}
\item   (i) $n = p^k$ with $p$ prime $k \geq m^2 + 1$, 
\item   (ii) $n = p^k$ with $p \nmid q$ prime and $k < m^2 + 1$,
\item (iii) $n \in \{p^k : k < m^2 + 1, p \mid q\}$
\end{itemize}
and call the corresponding contributions $S_1, S_2, S_3$. 

In case (i), we use the bounds \eqref{eqn:vonMangoldtBound} together with  $\alpha(n)\ll 1/\sqrt{n}$ to see that this portion of the $n$-sum is 
$$\ll \sum_{p \leq x} \sum_{ k \geq  m^2 +1}   p^{-k/(m^2 + 1)} \log p \ll \sum_{p \leq x}\frac{\log p}{p} \ll \log x \ll \log q,$$
so that  \eqref{eqn:average_weight} provides the bound $S_1 \ll (\log q)^2$. 

In case (iii), we use again \eqref{eqn:vonMangoldtBound} to see that the $n$-sum is $\ll \omega(q) \ll \log q$, and we obtain $S_3 \ll  (\log q)^2$ by \eqref{eqn:average_weight}. 

It remains to treat case (ii). By \eqref{eqn:schura}, it suffices to bound
$$S_{2,1}(k) := \frac{1}{V(q)}\sum_{\pi\in\mathcal{F}(q)}\frac{1}{L(1,\pi,\Ad)}\Big|\sum_{\substack{ p^k\leq x\\ p \nmid q} }\alpha_1(p^k)\lambda_{\pi}(p^k)\Big|^2$$
and 
$$S_{2,j}(k) := \frac{1}{V(q)}\sum_{\pi\in\mathcal{F}(q)}\frac{1}{L(1,\pi,\Ad)}\Big|\sum_{ \substack{p^k\leq x \\ p \nmid q}}\alpha_1(p^k)B_{\pi}(1, \ldots, 1, p, 1, \ldots, 1, p^{k-j})\Big|^2 $$
for $2 \leq j \leq \min(m, k)$ and $2 \leq k < m^2 + 1$, where the entry $p$ is at position $m-j$, and $\alpha_1(n)\ll n^{-1/2}\log n$. In either case, we can apply  \cref{thm:MVTgeneral} to obtain
$$S_{2,j}(k) \ll \sum_{p^k \leq x} \frac{(\log p)^2}{p^k}  \ll   (\log x)^2 \ll (\log q)^2$$
for all $1 \leq j \leq \min(m, k)$ and $1 \leq k < m^2 + 1$. 
\end{proof}

We finish the section with the \textbf{proof of \cref{thm:MVTgeneral}.} 
The proof uses the technology established in \cite[Theorem 4]{Blomer_density}, and we assume some familiarity with that paper.  In the expression
$$ \frac{1}{V(q)}\sum_{\pi \in \mathcal{F}(q)} \frac{1}{L(1, \pi, \text{Ad})}\Big|\sum_{\substack{w(\textbf{n} \leq x\\ (\textbf{n}, q) = 1}} \beta(\textbf{n}) B_{\pi}(\textbf{n})\Big|^2,$$
we add by positivity the rest of the non-residual spectrum, getting a full spectral expression of level $q$ that we denote by $\int_{(q)} d\varpi$, as in \cite[Section 5]{Blomer_density}. We detect the condition that the Laplacian eigenvalue is in $I$ by a finite collection of inner products $\langle E_j, W_{\varpi} \rangle$ for fixed compactly supported test functions $E_1, \ldots, E_r : \Bbb{R}_{>0}^{n-1} \rightarrow \Bbb{R}$ and Whittaker functions $W_{\varpi}$ as in \cite[Lemma 6]{Blomer_density} with $Z=1$.  Opening the square and recalling \eqref{eqn:sothat} and \cite[(5.3)]{Blomer_density}, we are left with bounding
\begin{equation}\label{sumj}
\begin{split}
\sum_j\Big| \sum_{\substack{w(\textbf{n}), w(\tilde{\textbf{n}}) \leq x\\ (\textbf{n}, q) = (\tilde{\textbf{n}}, q) = 1}} \beta(\textbf{n}) \beta(\tilde{\textbf{n}}) \mathcal{S}_{E_j}(N, \tilde{N})\Big|,
 \end{split}
\end{equation}
where
\begin{displaymath}
\begin{split}
&\mathcal{S}_E(N, \tilde{N}) :=  \int_{(q)} \overline{A_{\varpi}(N)} A_{\varpi}(\tilde{N}) |\langle W_{\varpi}, E\rangle|^2 d\varpi\\
&N = (n_1, \ldots, n_{m-1}), \quad \tilde{N} = (\tilde{n}_1, \ldots, \tilde{n}_{m-1} ),
\end{split}
\end{displaymath}
and $A_{\varpi}$ are the $L^2$-normalized Fourier coefficients of an automorphic constituent $\varpi$. (Here we used that $A_{\varpi}(n_{m-1}, \ldots, n_1) = \overline{A_{\varpi}(n_1, \ldots, n_{m-1})}$ to slightly simplify the notation.)
The proof of \cref{thm:MVTgeneral} follows now from the following lemma.

\begin{lemma}\label{35} There exists a constant $c_0 = c_0(E) > 0$   such that the following holds: if $w(\textbf{n}), w(\tilde{\textbf{n}}) \leq c_0q$, then  $\mathcal{S}_E(N, N') \ll \delta_{N = N'}.$ 
\end{lemma}

\begin{proof} We   apply \cite[Lemma 7]{Blomer_density} to $\mathcal{S}_E(N, N')$ and translate it into a sum over Kloosterman sums that can be estimated trivially as  
$$\sum_{w \in W} \sum_{v \in V} \sum_{c \in \Bbb{N}^{m-1}} \frac{|S_{q, w}^v(N, \tilde{N}, c)|}{c_1 \cdots c_{m-1} } \frac{1}{{\rm y}(\iota(N \cdot  {^w\!\tilde{N}}) c^{\ast}  )^{\eta}}  \int_{\tilde{T}(\Bbb{R})} \int_{U_w(\Bbb{R})} |E\big({\rm y}(\iota(N \cdot  {^w\!\tilde{N}}) c^{\ast} w x y) \big)E({\rm y}(y))| \, dx\, d^{\ast} y. $$
The Kloosterman sums $S_{q, w}^v(N, \tilde{N}, c)$ are defined in \cite[(4.2)]{Blomer_density}; all notation is explained in \cite[Section 2]{Blomer_density}. Specifically,  $W$ is the Weyl group, $V$ is a finite set of $2^m$ elements, the maps $\iota : \Bbb{R}_{>0}^{m-1} \rightarrow \GL_m(\Bbb{R})$ and ${\rm y} : \GL_m(\Bbb{R}) \rightarrow \Bbb{R}_{>0}^{m-1}$ are defined in \cite[(2.4), (2.5)]{Blomer_density},  the vector $^wy \in \Bbb{R}^{m-1}$  for   $y \in \Bbb{R}^{m-1}$ and $w \in W$ is given in \cite[(2.6)]{Blomer_density}, for    $c \in \Bbb{N}^{m-1}$ the diagonal matrix $c^{\ast}$ is defined in the display before \cite[(2.8)]{Blomer_density}, $\tilde{T}(\Bbb{R})$ is the set of $m$-dimensional diagonal matrices with positive entries and 1 in the bottom right corner, and $U_w = w^{-1} U^{\top} w \cap U$ where $U$ is the set of   unipotent upper triangular matrices. 

The trivial Weyl element $w = \text{id}$ contributes
\begin{equation*}
O(\delta_{N = \tilde{N}}),
\end{equation*}
since the Kloosterman sum vanishes unless $c_1 = \ldots = c_{m-1} = 1$ and $N = \tilde{N}$, so that ${\rm y}(\iota(N \cdot  {^w\!\tilde{N}}) c^{\ast} ) = 1$, and in addition $U_{\text{id}}(\Bbb{R}) = \{1\}$.

We want to show that the contribution of all other Weyl elements vanishes. To this end, we argue as in \cite[(7.2)]{Blomer_density} and determine upper bounds for the moduli $c_1, \ldots, c_{m-1}$ that ensure that the integral over $\tilde{T}(\Bbb{R}) \times U_w(\Bbb{R})$  does not vanish. These upper bounds will contradict the divisibility conditions of the moduli that we recall below. 

The integral over $\tilde{T}(\Bbb{R}) \times U_w(\Bbb{R})$ can only be non-zero if
$${\rm y}(\iota(N \cdot  {^w\!\tilde{N}}) c^{\ast} w x y) \in \text{supp}(E)$$
for some $x \in U_w(\Bbb{R})$ and some $y$ such that ${\rm y}(y) \in \text{supp}(E)$.  We now apply \cite[Lemma 1]{Blomer_density}  with $B = N  \cdot  {^w\!\tilde{N}} \in \Bbb{Q}_{>0}^{m-1}$ to obtain the bound
\begin{equation}\label{boundc}
c_j \ll_E \prod_{i=1}^{m-1} B_i^{s(i, j)}, \quad s(i, j) = \frac{1}{m} \begin{cases} i(n-j), & i \leq j,\\ j(n-i), & i > j.\end{cases} 
\end{equation}
To proceed further, we need to explicate the coordinates of $B$, and to this end we need to write down explicitly the coordinates of $${^w\!\tilde{N}} = \Big( \frac{\tilde{n}_1 \cdots \tilde{n}_{m - w(m- j + 1)}}{\tilde{n}_1 \cdots \tilde{n}_{m - w(m-j)}} \Big)_{1 \leq j \leq m-1} =: N' = (n_1', \ldots, n_{m-1}'),$$ say, as defined in \cite[(2.6)]{Blomer_density} where the permutation matrix $w = (w_{ij}) \in W$ is identified with the permutation $i \mapsto j$ for $w_{ij} = 1$.   Let $$w = \left(\begin{matrix} &&& I_{d_1} \\ && I_{d_2} & \\ & \ddots && \\ I_{d_r} &&&\end{matrix}\right)$$ be an admissible Weyl element.  As a permutation, this element sends $1 \mapsto m-d_1 + 1$, $2 \mapsto m-d_1 + 2, \ldots, d_1 + 1 \mapsto m - d_1 - d_2 + 1, \ldots, d_1 + \cdots + d_{r-1} + 1 \mapsto 1, \ldots, m \mapsto d_r$. 
 We write
$$N' = (N_{r},  \ldots,  N_{1})$$
where $N_i$ is a vector of length $d_i$ and for $i \leq r-1$ and $N_r$ has length $d_r - 1$. For $i \geq 0 $ write $D_i = d_0 + d_1 + \ldots + d_i$ with $d_0 = 0$. We have
$$N_{i} = \Big(\tilde{n}_{D_{i-1} + 1} \cdots \tilde{n}_{D_{i+1} - 1}, \frac{1}{\tilde{n}_{D_{i-1} + 1}}, \cdots, \frac{1}{\tilde{n}_{D_i - 1}}\Big), \quad i \leq r-1$$
  and
$$N_r = \Big(  \frac{1}{\tilde{n}_{D_{r-1} + 1}}, \cdots, \frac{1}{\tilde{n}_{D_r - 1}}\Big)$$
(which is the empty vector if $d_ r = 1$). We compute
\begin{displaymath}
\begin{split}
 \prod_{i = 1}^{m-1} (n'_{m-i})^i  = & \frac{1}{\tilde{n}_{d_1 - 1}} \cdots \frac{1}{\tilde{n}_1^{d_1 - 1}} (\tilde{n}_1 \cdots \tilde{n}_{d_1 + d_2 - 1})^{d_1} \\
 &\times \frac{1}{\tilde{n}_{d_1 + d_2 - 1}^{d_1 + 1} } \cdots \frac{1}{\tilde{n}_{d_1 + 1}^{d_1 + d_2 - 1}} (\tilde{n}_{d_1+ 1} \cdots \tilde{n}_{d_1 + d_2 + d_3 - 1})^{d_1 + d_2}\\
 & \times \cdots \times  \frac{1}{\tilde{n}^{d_1 + \ldots + d_{r-1} + 1}_{d_1 + \cdots + d_r - 1}} \cdots \frac{1}{\tilde{n}^{d_1 + \ldots + d_r - 1}_{d_1 + \cdots + d_{r-1} + 1}}\\
  = & \prod_{i= 1}^{m-1} (\tilde{n}_i)^i \cdot \prod_{i = m - d_1+1}^{m - 1} (\tilde{n}_i)^{-m}
\end{split}
\end{displaymath}
and similarly 
$$\prod_{i=1}^{m-1} (n'_{i})^{i} = \prod_{i=1}^{m-1}  (\tilde{n}_{m-i}')^i  \prod_{i = 1}^{d_1 - 1} \tilde{n}_i^{-m}.$$
With this information, we return to \eqref{boundc} and obtain
\begin{displaymath}
\begin{split}
c_j &\ll \prod_{i=1}^{m-1} (n_i n'_i)^{s(i, j)} \leq \prod_{i=1}^{m-1} n_i^{i(m-j)/m} \prod_{i=1}^{m-1} (n'_i)^{j(m-i)/m}\\
&  = \prod_{i=1}^{m-1} n_i^{i(m-j)/m} \prod_{i=1}^{m-1} \tilde{n}_i^{ji/n} \prod_{i = m - d_1+1}^{m- 1} \tilde{n}_i^{-j}  \leq  w(\textbf{n})^{\frac{m-j}{m}} w(\tilde{\textbf{n}})^{\frac{j}{m}} \leq c_0q. 
\end{split}
\end{displaymath}
From \cite[(4.3)]{Blomer_density} we know that
$$q \mid c_1, \ldots, c_{n-d_1}.$$
If $w \not = \text{id}$, then $d_1 < n$, so $q \mid c_1 \ll c_0q$.  This is a contradiction if $x \leq c_0 q$ for a sufficiently small constant $c_0$.   We choose $\Cr{Blomer1}$ in \cref{thm:MVTgeneral} to equal $\min_j c_0(E_j)$ for the finite collection of test functions $E_j$ appearing in \eqref{sumj}.
\end{proof}

\section{Proof of \texorpdfstring{\cref{thm:ZDE,thm:ZDE2}}{Theorems \ref*{thm:ZDE} and \ref*{thm:ZDE2}}}
\label{sec:ZDE}

\subsection{Proof of \cref{thm:ZDE}}
Let $q$ be a sufficiently large integer and $0<\delta<\frac{1}{2}$. Let $A$ be a sufficiently large constant (in terms of $\delta$). For instance, we can choose $A = (3-\delta)/(2\delta)$.  We regard $\delta$ and $A$ as fixed, and all implied constants may depend on these numbers.  Let $\epsilon = 1/(2m^2)$ and
\begin{equation}
\label{eqn:Xchoice}
X=\Cr{Blomer1}q,\qquad Y=q^{1-\delta},\qquad 1\leq T\leq q^{ (1-\epsilon)/m}.
\end{equation}
 If $T> q^{(1-\epsilon)/m}$, then \cref{thm:ZDE} is trivial compared to \eqref{eqn:RvonM} when $\sigma < 1-40\epsilon/143$ or \eqref{eqn:HT_GLn} when $\sigma\geq 1-40\epsilon/143$.  If $T\leq q^{(1-\epsilon)/m}$ and $\sigma\leq 1/2+1/\log q$, then \cref{thm:ZDE} is trivial compared to what follows from \eqref{eqn:RvonM}.  Therefore, in what follows, we may assume that $\sigma> 1/2+1/\log q$ and $T\leq q^{(1-\epsilon)/m}$.

Given $\pi\in\mathcal{F}(q)$, let
\begin{equation*}
M_X(s,\pi)= \sum_{n\leq X}\frac{\mu_{\pi}(n)}{n^s},\qquad 
LM_X(s,\pi)= L(s,\pi) M_X(s,\pi)
\end{equation*}
and recall the (polynomial) bound \eqref{eqn:mu_bound} for $\mu_{\pi}(n)$.  
By \cite[Proposition 5.7]{IK}, if $t\in[-T,T]$, then
\[
|\{\rho=\beta+i\gamma\colon L(\rho,\pi)=0,~\beta\geq 0,~|\gamma-t|\leq 1\}|\ll \log C(\pi,t)\ll  \log q.
\]
The box $[\sigma,1]\times[-T,T]$ is covered by $O(T)$ boxes of the form $[\sigma,1]\times[y-10\log Y,y+10\log Y]$, each containing $O((\log q)^2)$ zeros.  If we write $n_{\pi}=n_{\pi}(\sigma,T)$ for the number of such smaller boxes containing at least one zero $\rho$ of $L(s,\pi)$, then
\begin{equation}
\label{eqn:little_n}
N_{\pi}(\sigma,T)\ll (\log q)^2 n_{\pi}.
\end{equation}

Let $\rho=\beta+i\gamma$ be a zero of $L(s,\pi)$ with $ 1/2+1/\log q<\sigma\leq \beta$ and $|\gamma|\leq T$, in which case $LM_X(\rho,\pi)=0$.   Let $\int_{(c)}$ denote a contour integral along the line $\re(s)=c$.  The existence of such a $\rho$ implies that
\begin{equation}
\begin{aligned}
\label{eqn:zero_detect_1}
e^{-1/Y}&=\frac{1}{2\pi i}\int_{(1-\beta+A)}(1-LM_X(\rho+w,\pi))\Gamma(w)Y^{w}dw\\
&+\frac{1}{2\pi i}\int_{(\frac{1}{2}-\beta)}LM_X(\rho+w,\pi)\Gamma(w)Y^{w}dw.
\end{aligned}
\end{equation}
If $\re(w) = 1-\beta+A$ with $A$ sufficiently large, then
\begin{align*}
	|1-LM_X(\rho+w,\pi)|&= |L(\rho+w,\pi)|\cdot|L(\rho+w,\pi)^{-1}-M_X(\rho+w,\pi)|\\
	&= |L(1+A+i\im(\rho+w),\pi)|\cdot\Big|\sum_{n>X}\frac{\mu_{\pi}(n)}{n^{1+A+i\im(\rho+w)}}\Big| \ll X^{1-A},
\end{align*}
so that by Stirling's formula we obtain
\begin{equation*}
\frac{1}{2\pi i}\int_{(1-\beta+A)}(1-LM_X(\rho+w,\pi))\Gamma(w)Y^{w}dw\ll \frac{Y^{1-\sigma+A}}{X^{A-1}} \ll \frac{1}{Y}
\end{equation*}
for $A$ sufficiently large by \eqref{eqn:Xchoice}. 

We now turn towards the second term on the right hand side of \eqref{eqn:zero_detect_1}. Again, by Stirling's formula, trivial bounds for $M_X(1/2 + it, \pi)$, and the convexity bound \eqref{eqn:convexity} for  $L(1/2 + it, \pi)$ (recall that $T \leq q^{ (1-\epsilon)/q}$), we have
\begin{equation*}
\begin{aligned}
&\frac{1}{2\pi i}\int_{\substack{\re(w)=\frac{1}{2}-\beta \\ |\im(w)|\geq 10\log Y}}LM_X(\rho+w,\pi)\Gamma(w)Y^{w}dw \ll \frac{1}{Y}.
\end{aligned}
\end{equation*}
Combining the previous two bounds with 
 the asymptotic $e^{-1/Y}=1+O(1/Y)$ and  $\beta\geq\sigma>\frac{1}{2}+\frac{1}{\log q}$, we arrive at
\begin{equation}
\begin{aligned}
\label{eqn:zero_detect_2}
&\ll \int_{\frac{1}{2}-\beta-10i\log Y}^{\frac{1}{2}-\beta+10i\log Y}LM_X(\rho+w,\pi)\Gamma(w)Y^{w}dw\\
&\ll Y^{\frac{1}{2}-\sigma}\Big((\log q)\int_{|\gamma-v|\leq 1}|LM_X(\tfrac{1}{2}+iv,\pi)|dv+\int_{1\leq |\gamma-v|\leq 10\log Y}|LM_X(\tfrac{1}{2}+iv,\pi)|dv\Big).
\end{aligned}
\end{equation}
The factor $\log q$ in the second line is included as an upper bound $\Gamma(w)$ at $w = \frac{1}{\log q} + iv$ for small $v$.

We detect zeros in the rectangles $[\sigma,1]\times[y-10\log Y,y+10\log Y]$ using \eqref{eqn:zero_detect_2}, thus obtaining
\begin{equation}
\label{eqn:zero_detect_4}
n_{\pi}\ll Y^{\frac{1}{2}-\sigma}\Big(T(\log q)\sup_{|t_0|\leq T}\int_{|t_0-v|\leq 1}|LM_X(\tfrac{1}{2}+iv,\pi)|dv+\int_{-T-10\log Y}^{T+10\log Y}|LM_X(\tfrac{1}{2}+iv,\pi)|dv\Big).
\end{equation}
Using \eqref{eqn:little_n} and \eqref{eqn:zero_detect_4}, we sum over $\pi\in\mathcal{F}(q)$ and insert the weights $L(1,\pi,\Ad)^{-1}$:
\begin{equation}
\label{eqn:bound_int}
\begin{aligned}
\sum_{\pi\in\mathcal{F}(q)}\frac{N_{\pi}(\sigma,T)}{L(1,\pi,\Ad)}&\ll (\log q)^2\sum_{\pi\in\mathcal{F}(q)}\frac{n_{\pi}}{L(1,\pi,\Ad)}\\
&\ll (\log q)^2 Y^{\frac{1}{2}-\sigma}\int_{-T-10\log Y}^{T+10\log Y}\sum_{\pi\in\mathcal{F}(q)}\frac{|LM_X(\tfrac{1}{2}+iv,\pi)|}{L(1,\pi,\Ad)}dv\\
&+T(\log q)^3
Y^{\frac{1}{2}-\sigma}\sup_{|t_0|\leq T}\int_{|t_0-v|\leq 1}\sum_{\pi\in\mathcal{F}(q)}\frac{|LM_X(\tfrac{1}{2}+iv,\pi)|}{L(1,\pi,\Ad)}dv.
\end{aligned}
\end{equation}
The Cauchy--Schwarz inequality and our choice of $Y$ imply that \eqref{eqn:bound_int} is
\[
\ll T(\log q)^3 Y^{\frac{1}{2}-\sigma}\sup_{|t|\leq T+10\log q}\Big(\sum_{\pi\in\mathcal{F} (q)}\frac{|L(\tfrac{1}{2}+it,\pi)|^2}{L(1,\pi,\Ad)}\Big)^{\frac{1}{2}}  \Big(\sum_{\pi\in\mathcal{F} (q)}\frac{|M_X(\tfrac{1}{2}+it,\pi)|^2}{L(1,\pi,\Ad)}\Big)^{\frac{1}{2}}.
\]
\cref{cor:2ndmoment,cor:MVT2} yield now
\[
\frac{1}{V(q)} \sum_{\pi\in\mathcal{F}(q)}\frac{N_{\pi}(\sigma,T)}{L(1,\pi,\Ad)}\ll 
 \begin{cases}
 TY^{\frac{1}{2}-\sigma}F(q)(\log q)^4&\mbox{if $m\geq 3$,}\\
 TY^{\frac{1}{2}-\sigma}F(q)(\log q)^5&\mbox{if $m=2$.}
 \end{cases}
\]
In view of \eqref{eqn:Xchoice} concludes the proof.

\subsection{Proof of \cref{thm:ZDE2}}

The proof of \cref{thm:ZDE2} is almost identical, so we can be brief. In \eqref{eqn:Xchoice} we make the same choices for $X$,  $Y$ and $T$ with $Q$ in place of $q$. Let $\epsilon=1/(2m^2)$.  If $T>Q^{(1-\epsilon)/m}$ or $\sigma\leq 1/2 + 1/\log Q$, then \cref{thm:ZDE2} is trivial compared to \eqref{eqn:HT_GLn} or \eqref{eqn:RvonM}.  Otherwise, {\it mutatis mutandis}, we obtain
\begin{displaymath}
\begin{split}
&\sum_{\pi\in\mathcal{G}(Q)}\frac{N_{\pi}(\sigma,T)}{L(1,\pi,\Ad)} \\
&\ll T(\log Q)^3 Y^{\frac{1}{2}-\sigma}\sup_{|t|\leq T+O(\log Q)}\Big(\sum_{\pi\in\mathcal{G} (Q)}\frac{|L(\tfrac{1}{2}+it,\pi)|^2}{L(1,\pi,\Ad)}\Big)^{\frac{1}{2}}  \Big(\sum_{\pi\in\mathcal{G} (Q)}\frac{|M_X(\tfrac{1}{2}+it,\pi)|^2}{L(1,\pi,\Ad)}\Big)^{\frac{1}{2}}.
\end{split}
\end{displaymath} 
The analogue of the large sieve inequality \cref{thm:MVTgeneral} is now played by  \cite[Theorem 5]{Jana_newvectors} which  under the current assumptions provides the bounds
$$ \sum_{\pi\in\mathcal{G} (Q)}\frac{|L(\tfrac{1}{2}+it,\pi)|^2}{L(1,\pi,\Ad)} \ll |\mathcal{G}(Q)| \log Q, \qquad   \sum_{\pi\in\mathcal{G} (Q)}\frac{|M_X(\tfrac{1}{2}+it,\pi)|^2}{L(1,\pi,\Ad)} \ll  |\mathcal{G}(Q)| \log Q,
$$
analogously to the proofs of \cref{cor:2ndmoment,cor:MVT2}.

\section{Proof of \cref{thm:order_of_vanishing}}

Fix a function $\psi\colon\R\to[0,\infty)$ that is infinitely differentiable, even, nonnegative, and compactly supported in $[-1,1]$.  Normalize $\psi$ so that $\psi(0)=1$.  Consider the Laplace transform
\[
\Psi(z) = \int_{-\infty}^{\infty}\psi(t)e^{-tz}dt.
\]
Since $\psi$ is even, we have that $\Psi(z)=\Psi(-z)$.  We may choose $\psi$ so that
\begin{equation}
\label{eqn:Psi_lower}
\re(\Psi(z))\geq 0,\qquad |\re(z)|\leq 1.
\end{equation}
Such a $\psi$ satisfying \eqref{eqn:Psi_lower} was constructed for the purpose of obtaining lower bounds for discriminants of number fields in \cite{Poitou}.  For any integer $k\geq 1$, repeated integration by parts yields the bound
\begin{equation}
\label{eqn:Psi_upper}
|\Psi(z)|\ll_{\Psi,k}e^{|\re(z)|}(1+|z|)^{-k}.
\end{equation}

Let $\pi\in\mathcal{F}(q)$ and $x\geq 3$.  It follows from the argument principle, \eqref{eqn:FE}, and the identity $\Psi(z)=\Psi(-z)$ that
\begin{align*}
&\sum_{\rho}\Psi((\rho-\tfrac{1}{2})\log x)\\
&=\frac{1}{2\pi i}\int_{(\frac{1}{2}+\frac{1}{\log x})}\frac{\Lambda'}{\Lambda}(\tfrac{1}{2}+s,\pi)\Psi(s\log x)ds-\frac{1}{2\pi i}\int_{(-\frac{1}{2}-\frac{1}{\log x})}\frac{\Lambda'}{\Lambda}(\tfrac{1}{2}+s,\pi)\Psi(s\log x)ds\\
&=\frac{1}{2\pi i}\int_{(\frac{1}{2}+\frac{1}{\log x})}\Big(\frac{\Lambda'}{\Lambda}(\tfrac{1}{2}+s,\pi)+\frac{\Lambda'}{\Lambda}(\tfrac{1}{2}+s,\tilde{\pi})\Big)\Psi(s\log x)ds.
\end{align*}
A standard calculation using Laplace inversion, \eqref{eqn:completed_L-function}, \eqref{eqn:Psi_upper}, the identity $\psi(0)=1$, and  Stirling's formula shows that the preceding display equals
\begin{align*}
&\frac{1}{\log x}\Big(\log \frac{q}{\pi^m} - 2\re\sum_{n=1}^{\infty}\frac{a_{\pi}(n)\Lambda(n)}{\sqrt{n}}\psi\Big(\frac{\log p^k}{\log x}\Big)\Big)\\
&+\frac{1}{\log x}\sum_{j=1}^m\int_{-\infty}^{\infty}\Big(\frac{\Gamma'}{\Gamma}\Big(\frac{\frac{1}{2}+it+\kappa_{\pi}(j)}{2}\Big)+\frac{\Gamma'}{\Gamma}\Big(\frac{\frac{1}{2}+it+\overline{\kappa_{\pi}(j)}}{2}\Big)\Big)\Psi(it\log x)dt\\
&=\frac{\log (q\pi^{-m})}{\log x} - \frac{2}{\log x} \re\sum_{n=1}^{\infty}\frac{a_{\pi}(n)\Lambda(n)}{\sqrt{n}}\psi\Big(\frac{\log n}{\log x}\Big) +O\Big(\frac{1+ \log(C(\pi)/q)}{\log x}\Big).
\end{align*}
We now isolate the contribution from the zeros on the left-hand side that are close to $s=\frac{1}{2}$, take real parts and apply  \eqref{eqn:Psi_lower}, thus obtaining
\begin{align*}
\mathop{\mathrm{ord}}_{s=1/2}L(s,\pi)&\ll 
\sum_{\substack{\rho = \beta+i\gamma \\ |\beta-\frac{1}{2}|<\frac{1}{\log x}}}\re(\Psi((\rho-\tfrac{1}{2})\log x)) \\
&\ll \frac{\log q}{\log x}+\sum_{\substack{\rho = \beta+i\gamma \\ |\beta-\frac{1}{2}|\geq\frac{1}{\log x}}}|\Psi((\rho-\tfrac{1}{2})\log x)| + \frac{2}{\log x}\Big|\sum_{n=1}^{\infty}\frac{a_{\pi}(n)\Lambda(n)}{\sqrt{n}}\psi\Big(\frac{\log n}{\log x}\Big) \Big|.
\end{align*}
\begin{proof}[Proof of \cref{thm:order_of_vanishing}, weighted version]
Using \cref{newcor} and \eqref{eqn:average_weight} along with the Cauchy-Schwarz inequality, we obtain
\[
\frac{1}{V(q)}\sum_{\pi\in\mathcal{F}(q)}\frac{\mathop{\mathrm{ord}}_{s=1/2}L(s,\pi) }{L(1,\pi,\Ad)}\ll \frac{\log q }{\log x}+\frac{1}{V(q)}\sum_{\pi\in\mathcal{F}(q)}\frac{1}{L(1,\pi,\Ad)}\sum_{\substack{\rho = \beta+i\gamma \\ |\beta-\frac{1}{2}|\geq\frac{1}{\log x}}}|\Psi((\rho-\tfrac{1}{2})\log x)|.
\]
By the symmetry of the zeros as a consequence of the functional equation (note that $\pi\in\mathcal{F}(q)$ implies $\tilde{\pi}\in\mathcal{F}(q)$),   it is enough to consider the zeros $\beta+i\gamma$ with $\beta\geq \frac{1}{2}+\frac{1}{\log x}$ and $\gamma\geq 0$.  We subdivide the region $[\frac{1}{\log x},\frac{1}{2}]\times[0,\infty)$ into small squares of the form
\[
\Big[\frac{a}{\log x},\frac{a+1}{\log x}\Big]\times\Big[\frac{b}{\log x},\frac{b+1}{\log x}\Big],\qquad a\in[1,\lceil\log x\rceil]\cap\Z,\qquad b\in\mathbb{N}.
\]
Using \cref{thm:ZDE} and \eqref{eqn:Psi_upper} with $k=m+4$, we obtain the estimate (cf. \cite[Section 4]{KM_J0})
\begin{align*}
&\frac{1}{V(q)}\sum_{\pi\in\mathcal{F}(q)}\frac{1}{L(1,\pi,\Ad)}\sum_{\substack{\rho = \beta+i\gamma \\ |\beta-\frac{1}{2}|\geq\frac{1}{\log x}}}|\Psi((\rho-\tfrac{1}{2})\log x))|\\
&\ll \sum_{1\leq a\leq \log x}\,\sum_{b\geq 0}\Big(1+\frac{b+1}{\log x}\Big)^A q^{-B a/\log x}F(q)(\log q)^5 e^{a+1}(b+1)^{-k}\ll \frac{F(q)(\log q)^5}{q^{B/\log x}}.
\end{align*}
The weighted average in \cref{thm:order_of_vanishing} follows once we choose
\[
x=\exp\Big(B\frac{\log q}{\log(F(q)(\log q)^5)}\Big).\qedhere
\]
\end{proof}

\begin{proof}[Proof of \cref{thm:order_of_vanishing}, unweighted version]
Alternatively, we can bound the sum over prime powers trivially using  \eqref{eqn:vonMangoldtBound} and get the unweighted average
 \begin{align*}
\frac{1}{V(q)}\sum_{\pi\in\mathcal{F}(q)}\mathop{\mathrm{ord}}_{s=1/2}L(s,\pi)\ll \frac{x^{\frac{1}{2}+\theta_m }+\log q }{\log x}+\frac{1}{V(q)}\sum_{\pi\in\mathcal{F}(q)}\sum_{\substack{\rho = \beta+i\gamma  \\ |\beta-\frac{1}{2}|\geq\frac{1}{\log x}}}|\Psi((\rho-\tfrac{1}{2})\log x)|.
\end{align*}
We use \cref{cor:ZDE} instead of \cref{thm:ZDE} to estimate the sum over zeros $\rho=\beta+i\gamma$ satisfying $|\beta-\frac{1}{2}|\geq1/\log x$, and we bound $R_{\mathcal{F}}(q)F(q)$ using \eqref{eqn:Li_Ad} and \eqref{eqn:F(q)_bound}.  Finally, we choose $\Cl[abcon]{xrange}>0$ to be a sufficiently small constant and $x = (\log q)^{\Cr{xrange}}$.
\end{proof}

\section{Proofs of \cref{thm:zeros_equidistr,thm:zeros_equidistr2}}

\subsection{Proof of \cref{thm:zeros_equidistr}}
Throughout this subsection, we let $q$ be a large prime.  With $B$ as in \cref{thm:ZDE}, let
\begin{equation}
\label{eqn:conditions}
r>2,\qquad T=(\log q)^r,\qquad 1<x\leq \exp\Big(\frac{B\log T}{3\log\log T}\Big),\qquad \delta = \frac{3\log \log T}{2B\log T}.
\end{equation}
Define $\mathop{\mathrm{sinc}}(0)=1$ and $\mathop{\mathrm{sinc}}(t)=(\sin t)/t$ otherwise.  Our proof requires two results.

\begin{lemma}
\label{lem:GonekLandau}
	Let $\pi\in\mathcal{F}(q)$.  If $\langle x\rangle$ is the integer closest to $x$ and $\epsilon>0$, then
\begin{displaymath}
\begin{split}
\frac{1}{V(q)}\sum_{\pi\in\mathcal{F}(q)}\frac{1}{L(1,\pi,\Ad)}&\frac{1}{N_{\pi}(T)}\sum_{|\gamma|\leq T}x^{\rho-\frac{1}{2}} \\
&= -\frac{T}{\pi\sqrt{x}}\mathop{\mathrm{sinc}}\Big(T\log\frac{x}{\langle x\rangle}\Big)\frac{\Lambda(\langle x\rangle)}{V(q)}\sum_{\pi\in\mathcal{F}(q)}\frac{a_{\pi}(\langle x\rangle)}{L(1,\pi,\Ad)}\frac{1}{N_{\pi}(T)}\\
&+O_{\epsilon}\Big(\frac{x^{\frac{1}{2}+\theta_m}\log(2x)}{T\log(qT)}+\frac{x^{\frac{1}{2}+\epsilon}}{T}+\frac{1}{T\sqrt{x}\log x}\Big).
\end{split}
\end{displaymath}
\end{lemma}
\begin{proof}
Working out the $q$-dependence in \cite[Lemma 2]{FSZ} using \cite[Proposition 5.7]{IK} (see also \cite[Proposition 1]{MurtyZaharescu}), we arrive at the asymptotic
\[
\sum_{|\gamma|\leq T}x^{\rho}=-\frac{T}{\pi}\mathop{\mathrm{sinc}}\Big(T\log\frac{x}{\langle x\rangle}\Big)\Lambda(\langle x\rangle)a_{\pi}(\langle x\rangle)+O_{\epsilon}\Big(x^{1+\theta_m}\log 2x+x^{1+\epsilon}\log qT+\frac{\log qT}{\log x}\Big).
\]
Divide by $V(q)L(1,\pi,\Ad)N_{\pi}(T)\sqrt{x}$, sum over $\pi\in\mathcal{F}(q)$, and apply \eqref{eqn:RvonM} and \eqref{eqn:average_weight}.
\end{proof}

\begin{lemma}
\label{lem:ZDE_application}
If $x$ and $T$ satisfy \eqref{eqn:conditions}, then
\[
\frac{1}{V(q)}\sum_{\pi\in\mathcal{F}(q)}\frac{1}{L(1,\pi,\Ad)}\frac{1}{N_{\pi}(T)}\sum_{|\gamma|\leq T}(x^{i\gamma}-x^{\rho-1/2})\ll \frac{1}{(\log qT)(\log T)^2}+\frac{(\log x)^2}{(\log qT)\log T}.
\]
\end{lemma}
\begin{proof}
Let $x$, $T$, and $\delta$ be as in \eqref{eqn:conditions}.  Note  that $0<\delta\log x<1$.  We find that
\begin{equation}
\label{eqn:ZDE_application1}
\begin{aligned}
&\frac{1}{V(q)}\sum_{\pi\in\mathcal{F}(q)}\frac{1}{L(1,\pi,\Ad)}\frac{1}{N_{\pi}(T)}\sum_{\substack{|\gamma|\leq T \\ |\beta-\frac{1}{2}|\geq \delta}}(x^{i\gamma}-x^{\rho-\frac{1}{2}})\\
&\ll \frac{1}{T\log qT}\frac{1}{V(q)}\sum_{\pi\in\mathcal{F}(q)}\frac{1}{L(1,\pi,\Ad)}\sum_{\substack{|\gamma|\leq T \\ \beta\geq \frac{1}{2}+\delta}}x^{\beta-\frac{1}{2}}\\
&\ll \frac{1}{T\log qT}\Big(\frac{x^{\delta}}{V(q)}\sum_{\pi\in\mathcal{F}(q)}\frac{N_{\pi}(\frac{1}{2}+\delta,T)}{L(1,\pi,\Ad)}+(\log x)\int_{\frac{1}{2}+\delta}^{1}x^{\sigma-\frac{1}{2}}\frac{1}{V(q)}\sum_{\pi\in\mathcal{F}(q)}\frac{N_{\pi}(\sigma,T)}{L(1,\pi,\Ad)}d\sigma\Big).
\end{aligned}
\hspace{-2mm}
\end{equation}
\cref{thm:ZDE} and our hypotheses on $x$ and $T$ ensure that \eqref{eqn:ZDE_application1} is $\ll (\log T)^{-2}(\log qT)^{-1}$.  On the other hand, we find that
\begin{equation}
\label{eqn:ZDE_application2}
\begin{aligned}
&\frac{1}{V(q)}\sum_{\pi\in\mathcal{F}(q)}\frac{1}{L(1,\pi,\Ad)}\frac{1}{N_{\pi}(T)}\sum_{\substack{|\gamma|\leq T \\ |\beta-\frac{1}{2}|< \delta}}(x^{i\gamma}-x^{\rho-\frac{1}{2}})\\
&\ll\frac{1}{T\log qT}\frac{1}{V(q)}\sum_{\pi\in\mathcal{F}(q)}\frac{1}{L(1,\pi,\Ad)}\sum_{\substack{|\gamma|\leq T \\ \frac{1}{2}<\beta<\frac{1}{2}+\delta}}|x^{i\gamma}(1-x^{\beta-\frac{1}{2}})+x^{i\gamma}(1-x^{\frac{1}{2}-\beta})|\\
&=\frac{1}{T\log qT}\frac{1}{V(q)}\sum_{\pi\in\mathcal{F}(q)}\frac{1}{L(1,\pi,\Ad)}\sum_{\substack{|\gamma|\leq T \\ \frac{1}{2}<\beta<\frac{1}{2}+\delta}}4\sinh\Big(\frac{1}{2}\Big(\beta-\frac{1}{2}\Big)\log x\Big)^2.
\end{aligned}
\end{equation}
If $0<\beta-\frac{1}{2}<\delta$, then $0<\frac{1}{2}(\beta-\frac{1}{2})\log x<\frac{1}{2}$.  Since $y<\sinh(y)<2y$ for all $0<y<\frac{1}{2}$, \eqref{eqn:ZDE_application2} is at most
\begin{align*}
&\ll \frac{(\log x)^2}{T(\log qT)V(q)}\sum_{\pi\in\mathcal{F}(q)}\frac{1}{L(1,\pi,\Ad)}\sum_{\substack{|\gamma|\leq T \\ 0<\beta-\frac{1}{2}<\delta}}\Big(\beta-\frac{1}{2}\Big)^2\\
&\ll \frac{(\log x)^2}{T(\log qT)} \int_0^{\delta}\frac{t}{V(q)}\sum_{\pi\in\mathcal{F}(q)}\frac{N_{\pi}(\frac{1}{2}+t,T)}{L(1,\pi,\Ad)}dt\\
&\ll \frac{(\log x)^2}{T(\log qT)} \Big(\int_0^{((m-1)Br + \frac{5}{B})\frac{\log\log q}{\log q}}+\int_{((m-1)Br + \frac{5}{B})\frac{\log\log q}{\log q}}^{\delta}\Big)\frac{t}{V(q)}\sum_{\pi\in\mathcal{F}(q)}\frac{N_{\pi}(\frac{1}{2}+t,T)}{L(1,\pi,\Ad)}dt.
\end{align*}
 We use \eqref{eqn:RvonM} to estimate the first integral and \cref{thm:ZDE} to estimate the second integral.  By \eqref{eqn:conditions}, their sum is $O((\log x)^2 (\log T)^{-1}(\log qT)^{-1})$.
\end{proof}

\begin{corollary}
\label{cor:GonekLandau+ZDE}
If $x$ and $T$ satisfy \eqref{eqn:conditions}, then
\begin{align*}
\frac{1}{V(q)}\sum_{\pi\in\mathcal{F}(q)}\frac{1}{L(1,\pi,\Ad)}&\frac{1}{N_{\pi}(T)}\sum_{|\gamma|\leq T}x^{i\gamma}\\
&=-\frac{T}{\pi\sqrt{x}}\mathop{\mathrm{sinc}}\Big(T\log\frac{x}{\langle x\rangle}\Big)\frac{\Lambda(\langle x\rangle)}{V(q)}\sum_{\pi\in\mathcal{F}(q)}\frac{a_{\pi}(\langle x\rangle)}{L(1,\pi,\Ad)}\frac{1}{N_{\pi}(T)}\\
&+O_r\Big(\frac{\log\log q}{(\log\log\log q)^2\log q}+\frac{1}{(\log q)^r \sqrt{x}\log x}\Big).
\end{align*}
\end{corollary}
\begin{proof}
This follows immediately from \eqref{eqn:conditions}, \cref{lem:GonekLandau}, and \cref{lem:ZDE_application}.
\end{proof}

Finally, we cite the Beurling--Selberg trigonometric approximation for the indicator function of an interval.

\begin{proposition}[{\cite[Chapter 1, Section 2]{Montgomery}}]
\label{prop:Beurling-Selberg}
Let $I\subseteq[0,1]$ be an interval, and let $M\geq 1$ be an integer.  There exist trigonometric polynomials
\[
S_{I,M}^{\pm}(\theta)=\sum_{|n|\leq M}a_{I,M}^{\pm}(n)e^{2\pi i n\theta}
\]
that satisfy the following properties.
\begin{enumerate}
	\item If $\theta\in[0,1]$, then $S_{I,M}^-(\theta)\leq \mathbf{1}_I(\theta)\leq S_{I,M}^+(\theta)$.
	\item If $|n|\leq M$, then $a_{I,M}^{\pm}(n)+a_{I,M}^{\pm}(-n)=2\re(S_{I,M}^{\pm}(n))$.
	\item If $|n|\leq M$, then $|a_{I,M}^{\pm}-\int_I e^{-2\pi i n t}dt|\leq 1/(M+1)$.
\end{enumerate}
\end{proposition}

\begin{proof}[Proof of \cref{thm:zeros_equidistr}]

Let $q$ be a large prime, let $\alpha\in\R^{\times}$, and recall \eqref{eqn:conditions}.  For $n\in\mathbb{N}$, write $x_n=e^{2\pi n |\alpha|}$.  It follows from \cref{prop:Beurling-Selberg} that
\begin{align*}
\frac{1}{N_{\pi}(T)}\sum_{\substack{|\gamma|\leq T \\ \{\alpha\gamma\}\in I}}1-|I|&\leq \sum_{1\leq|n|\leq M}\frac{a_{I,M}^+(n)}{N_{\pi}(T)}\sum_{\substack{|\gamma|\leq T \\ \{\alpha\gamma\}\in I}}e^{2\pi i n \alpha\gamma}+O\Big(\frac{1}{M}\Big)\\
&=2\re\Big(\sum_{1\leq n\leq M}\frac{a_{I,M}^+(n)}{N_{\pi}(T)}\sum_{\substack{|\gamma|\leq T \\ \{\alpha\gamma\}\in I}}x_n^{i\gamma}\Big)+O\Big(\frac{1}{M}\Big)
\end{align*}
We divide through by $V(q)L(1,\pi,\Ad)$ and sum over $\pi\in\mathcal{F}(q)$.  Using \eqref{eqn:average_weight}, we obtain
\begin{align*}
&\frac{1}{V(q)}\sum_{\pi\in\mathcal{F}(q)}\frac{1}{L(1,\pi,\Ad)}\Big(\frac{1}{N_{\pi}(T)}\sum_{\substack{|\gamma|\leq T \\ \{\alpha\gamma\}\in I}}1-|I|\Big)\\
&\leq 2\re\Big(\sum_{1\leq n\leq M}\frac{a_{I,M}^+(n)}{V(q)}\sum_{\pi\in\mathcal{F}(q)}\frac{1}{L(1,\pi,\Ad)N_{\pi}(T)}\sum_{\substack{|\gamma|\leq T \\ \{\alpha\gamma\}\in I}}x_n^{i\gamma}\Big)+O\Big(\frac{1}{M}\Big).
\end{align*}
\cref{prop:Beurling-Selberg} also yields a  corresponding lower bound, in which case
\begin{equation}
\label{eqn:equidistr_3}
\begin{aligned}
&\sup_{I\subseteq[0,1]}\Big|\frac{1}{V(q)}\sum_{\pi\in\mathcal{F}(q)}\frac{1}{L(1,\pi,\Ad)}\Big(\frac{1}{N_{\pi}(T)}\sum_{\substack{|\gamma|\leq T \\ \{\alpha\gamma\}\in I}}1-|I|\Big)\Big|\\
&\leq 2\sup_{I\subseteq[0,1]}\max_{\pm}\Big|\sum_{1\leq n\leq M}\frac{a_{I,M}^{\pm}(n)}{V(q)}\sum_{\pi\in\mathcal{F}(q)}\frac{1}{L(1,\pi,\Ad)N_{\pi}(T)}\sum_{\substack{|\gamma|\leq T \\ \{\alpha\gamma\}\in I}}x_n^{i\gamma}\Big|+O\Big(\frac{1}{M}\Big).
\end{aligned}
\end{equation}
If
\[
M= \frac{B\log T}{6\pi|\alpha|\log\log T}
\]
and $1\leq n\leq M$, then $x_n$ satisfies the hypotheses of \cref{cor:GonekLandau+ZDE}.  Using \eqref{eqn:average_weight} and the bound $|S_{I,M}^{\pm}(n)|\ll 1/n$ for $1\leq n\leq M$ that follows from \cref{prop:Beurling-Selberg}, we infer that \eqref{eqn:equidistr_3} is at most
\begin{align*}
&2\sup_{I\subseteq[0,1]}\max_{\pm}\Big|\frac{2T}{\pi}\re\sum_{1\leq n\leq M} \frac{a_{I,M}^{\pm}(n)}{\sqrt{x_n}}\mathop{\mathrm{sinc}}\Big(T\log\frac{x_n}{\langle x_n\rangle}\Big)\frac{\Lambda(\langle x_n\rangle)}{V(q)}\sum_{\pi\in\mathcal{F}(q)}\frac{a_{\pi}(\langle x_n\rangle)}{L(1,\pi,\Ad)}\frac{1}{N_{\pi}(T)}\Big| \\
&+ O\Big(\sum_{1\leq n\leq M}\frac{1}{n}\Big(\frac{\log\log q}{(\log\log\log q)^2\log q}+\frac{1}{(\log q)^r e^{\pi n \alpha}n|\alpha|}\Big)+\frac{1}{M}\Big).
\end{align*}
Per \eqref{eqn:RvonM},  \eqref{eqn:vonMangoldtBound}, and \eqref{eqn:average_weight}, the first line is $\ll_{\alpha,r} T/(T\log qT)\ll 1/\log q$.   The second line is $O_{\alpha,r}((\log\log\log q)/\log\log q)$, finishing the proof.
\end{proof}

\subsection{Proof of \cref{thm:zeros_equidistr2}} Once we replace $q$ with $Q$, $V(q)$ with $|\mathcal{G}(Q)|$, and $\mathcal{F}(q)$ with $\mathcal{G}(Q)$, the proof is the same as that of \cref{thm:zeros_equidistr}.

\bibliographystyle{abbrv}
\bibliography{BlomerThorner}

\end{document}